\documentclass[12pt]{amsart}
\title[Coherent state representations of a Lie group]{Coherent state representations \\of the holomorphic automorphism group \\of the tube domain \\over the dual of the Vinberg cone}
\author{Koichi Arashi}
\address{K. Arashi: Graduate School of Mathematics, Nagoya University, Chikusa-ku, Nagoya, 464-8602 Japan}
\email{m15005y@math.nagoya-u.ac.jp}
\usepackage{amssymb}
\usepackage{tikz} 
\usepackage{graphicx}
\usepackage{comment}
\numberwithin{equation}{section}
\theoremstyle{plain}
\newtheorem{theorem}{Theorem}[section]
\newtheorem{lemma}[theorem]{Lemma}
\newtheorem{proposition}[theorem]{Proposition}

\theoremstyle{definition}
\newtheorem{definition}[theorem]{Definition}

\theoremstyle{remark}
\newtheorem{remark}[theorem]{Remark}
\allowdisplaybreaks
\newcommand{\dt}{\left.\frac{d}{dt}\right|_{t=0}}

\newcommand{\one}{1}

\newcommand{\alpharbitgroup}{\alpha_0}
\newcommand{\matG}{G'}
\newcommand{\arbitline}{{L_0}}

\newcommand{\arbitgroup}{{G_0}}
\newcommand{\arbitalgebra}{{\mathfrak{g}_0}}
\newcommand{\arbitcenter}{Z_{\mathfrak{\mathfrak{g}}_0}}

\newcommand{\arbitmanifold}{{M_0}}

\newcommand{\ad}[1]{\mathop{\mathrm{ad}(#1)}}
\newcommand{\Int}[1]{\mathop{\mathrm{Int}}#1}
\newcommand{\Ad}[1]{\mathop{\mathrm{Ad}(#1)}}
\newcommand{\cAd}[1]{\mathop{\mathrm{Ad}^*(#1)}}

\begin{document}
\maketitle
\begin{abstract}
We classify all irreducible coherent state representations of the holomorphic automorphism group of the tube domain over the dual of the Vinberg cone.

\end{abstract}
\section{Introduction}
Let $\arbitgroup$ be a connected Lie group, and let $(\pi,\mathcal{H})$ be a unitary representation of $\arbitgroup$.
We regard the projective space $\mathbb{P}(\mathcal{H})$ as a (possibly infinite-dimensional) K\"{a}hler manifold.
We call a $\arbitgroup$-orbit of $\mathbb{P}(\mathcal{H})$ a {\it coherent state orbit} (CS orbit for short) if it is a complex submanifold of $\mathbb{P}(\mathcal{H})$, and we call $\pi$ a {\it coherent state representation} (CS representation for short) if there exists a CS orbit in $\mathbb{P}(\mathcal{H})$ that does not reduce to a point (see \cite[Definition 4.2]{Lisiecki95}).
In this case, we say that $\pi$ is {\it generic} if $\pi$ is irreducible and $\ker \pi$ is discrete.
By Lisiecki \cite{Lisiecki90}, the generic CS representations coincide with the irreducible highest weight representations with discrete kernels for a semisimple Lie group.
Thus CS representations can be considered as generalizations of the highest weight representations of semisimple Lie groups to a wider class of groups.
Also the generic CS representations of connected unimodular Lie groups were studied and classified by Lisiecki \cite{Lisiecki91}.
After this remarkable advance, CS representations were also studied in the setting of Lie groups which have compactly embedded Cartan subalgebras by Neeb \cite{Neeb}.

The purpose of the present article is to give classifications of irreducible CS representations and generic CS representations for a Lie group which has not been considered.
Let $\Omega_5$ be the dual cone of the Vinberg cone, and let $\mathcal{D}_5$ be the tube domain over $\Omega_5$.
Let $G$ be the identity component of the holomorphic automorphism group of $\mathcal{D}_5$.

In Section \ref{Generaltho}, we review the theory of CS representations studied in \cite{Lisiecki90, Lisiecki91, Lisiecki95}.
In Section \ref{theholomor}, we review the explicit description of $G$ studied in \cite{Geatti87, IK20}.
In Section \ref{CSorbitand}, we show that every generic CS representation of $G$ is unitarily equivalent with a unitarization of a holomorphic multiplier representation of $G$ over $\mathcal{D}_5$ or the complex conjugate representation of it.
In Section \ref{Holomorphi}, we review the classification of the unitarizations of holomorphic multiplier representations of $G$ over $\mathcal{D}_5$ studied in \cite{Arashi20}.
In Section \ref{GenericCSr}, we classify all generic CS representations of $G$.
In Section \ref{Irreducibl}, we classify all irreducible non-generic CS representations of $G$.
In Section \ref{Intertwini}, we consider intertwining operators between the external tensor product of a one-dimensional unitary representation of $\mathbb{R}_{>0}$ and an irreducible highest weight representation of $SL(2,\mathbb{R})\times SL(2,\mathbb{R})$ and the unitarizations of holomorphic multiplier representations of $G$ over $\mathcal{D}_5$.

The author would like to thank Professor H. Ishi for a lot of helpful advice on this paper.

\section{General theory of CS representations}\label{Generaltho}

Throughout this paper, for a Lie group, we denote its Lie algebra by the corresponding Fraktur small letter.

Let $\arbitgroup$ be a connected Lie group. For a $\arbitgroup$-equivariant holomorphic line bundle $\arbitline$ over a complex manifold $\arbitmanifold$, let us denote the natural representation of $\arbitgroup$ on the space $\Gamma^{hol}(\arbitmanifold,\arbitline)$ of holomorphic sections of $\arbitline$ by $\tau_{\arbitline}$.
We introduce a notion of unitarizability for $\tau_\arbitline$.
\begin{definition}\label{wesaythatr}
We say that the representation $\tau_{\arbitline}$ of $\arbitgroup$ is {\it unitarizable} if there exists a nonzero Hilbert space $\mathcal{H}\subset\Gamma^{hol}(\arbitmanifold,\arbitline)$ satisfying the following conditions:
\begin{enumerate}
\item[(i)]the inclusion map $\iota:\mathcal{H}\hookrightarrow \Gamma^{hol}(\arbitmanifold,\arbitline)$ is continuous with respect to the open compact topology of $\Gamma^{hol}(\arbitmanifold,\arbitline)$,
\item[(ii)] $\tau_{\arbitline}(g)\mathcal{H}\subset \mathcal{H}\quad(g\in \arbitgroup)$ and $\|\tau_{\arbitline}(g)s\|_\mathcal{H}=\|s\|_\mathcal{H}\quad(g\in \arbitgroup, s\in\mathcal{H})$.
\end{enumerate}
In this case, we call the subrepresentation $(\tau_{\arbitline},\mathcal{H})$ a {\it unitarization} of the representation $(\tau_{\arbitline},\Gamma^{hol}(\arbitmanifold,\arbitline))$ of $\arbitgroup$. 
\end{definition}
A Hilbert space $\mathcal{H}$ satisfying the condition (i) is a reproducing kernel Hilbert space.
We note that a Hilbert space giving a unitarization of $\tau_\arbitline$ is unique if it exists, and any unitarization is irreducible (see \cite{Ishi11, Kobayashi68, Kunze62}).
Thus we write $\pi_{\arbitline}$ instead of $(\tau_{\arbitline},\mathcal{H})$.
Let $(\pi,\mathcal{H})$ be a CS representation of $\arbitgroup$, and let $L$ be the natural holomorphic line bundle over $\mathbb{P}(\mathcal{H})$ such that the fiber over $[v]=\mathbb{C}v\in\mathbb{P}(\mathcal{H})$ is given by the dual space $[v]^*$. Then we can identify the dual space $\mathcal{H}^*$ with $\Gamma^{hol}(\mathbb{P}(\mathcal{H}),L)$.
By the following proposition, we can see that if $\pi$ is irreducible, then $\overline{\pi}$ is equivalent with $\pi_{\arbitline}$ for a $\arbitgroup$-equivariant holomorphic line bundle $\arbitline$ over a CS orbit in $\mathbb{P}(\mathcal{H})$, where $\overline{\pi}$ denotes the complex conjugate representation.
\begin{proposition}[{\cite[Proposition 2]{Lisiecki91}}]\label{Letpimathc1}
Suppose that $\pi$ is irreducible, and let $M\subset\mathbb{P}(\mathcal{H})$ be a CS orbit. Then the map $\mathcal{H}^*\rightarrow \Gamma^{hol}(M,L)$ given by the composition of the map $\mathcal{H}^*\rightarrow\Gamma^{hol}(\mathbb{P}(\mathcal{H}),L)$ and the restriction map $\Gamma^{hol}(\mathbb{P}(\mathcal{H}),L)\rightarrow\Gamma^{hol}(M,L)$ is injective.
\end{proposition}

Let $M$ be a CS orbit, let $\alpharbitgroup:\arbitgroup\times M\rightarrow M$ be the action of $\arbitgroup$ on $M$, and let $\arbitcenter$ be the center of $\arbitalgebra$.
When $\pi$ is generic, it holds that 
\begin{equation}\label{mathrmLiek}
\mathrm{Lie}(\ker\alpharbitgroup)=\arbitcenter, \end{equation}
where $\ker\alpharbitgroup=\{g\in\arbitgroup;\alpharbitgroup(g,x)=x\text{ for all }x\in M\}$.

Next let us see the relationship between CS orbits and coadjoint orbits.
Let $\mu_\pi:\mathbb{P}(\mathcal{H}^\infty)\rightarrow\arbitalgebra^*$ be a moment map defined by
\begin{equation*}
\langle x, \mu_\pi([v])\rangle=-i\frac{(d\pi(x)v,v)_\mathcal{H}}{(v,v)_\mathcal{H}}\quad(v\in\mathcal{H}^\infty\backslash\{0\}, x\in\arbitalgebra). 
\end{equation*}
Then the image of $M$ under $\mu_\pi$ coincides with a coadjoint orbit.
We note that $M$ has the natural structure of a K\"{a}hler manifold which is induced by the Fubini-Stdy metric on $\mathbb{P}(\mathcal{H})$.
As a consequence of this property, we have the following theorem.
\begin{theorem}[{\cite[Theorem 2.17]{RV85}}]
The isotropy subgroup of $\arbitgroup$ at any point of $\mu_\pi(M)$ is connected. In particular, the coadjoint orbit $\mu_\pi(M)$ is simply connected, and $\mu_\pi$ defines a diffeomorphism of $M$ onto the coadjoint orbit.
\end{theorem}

\section{The holomorphic automorphism group of the tube domain over the dual of the Vinberg cone}\label{theholomor}

Let
\begin{equation*}
V=\left\{\left[\begin{array}{ccc}x^\one&0&x^4\\0&x^2&x^5\\x^4&x^5&x^3\end{array}\right]\in M_3(\mathbb{R});x^1,\cdots ,x^5\in\mathbb{R}\right\},
\end{equation*}
and let $\Omega_5=V\cap \mathbb{P}(3,\mathbb{R})$, where $\mathbb{P}(3,\mathbb{R})$ denotes the homogeneous convex cone consists of all $3$-by-$3$ real positive-definite symmetric matrices. We consider the following Siegel domain $\mathcal{D}_5$ in $V_\mathbb{C}$:
\begin{equation*}
\mathcal{D}_5=\left\{z=\left[\begin{array}{ccc}z^\one&0&z^4\\0&z^2&z^5\\z^4&z^5&z^3\end{array}\right]\in V_\mathbb{C};\mathrm{Im}\,z\in \Omega_5\right\}.
\end{equation*}
Let $\mathrm{Aut}_{hol}(\mathcal{D}_5)$ be the holomorphic automorphism group of $\mathcal{D}_5$.
We note that $\mathcal{D}_5$ is holomorphically equivalent to a complex bounded domain, and $\mathrm{Aut}_{hol}(\mathcal{D}_5)$ has the unique structure of a Lie group compatible with the compact open topology.
Let $G$ be the identity component of $\mathrm{Aut}_{hol}(\mathcal{D}_5)$.
\begin{theorem}[\cite{Geatti87}, {\cite[Theorem 2.2]{IK20}}]\label{Thefollowi}
The following linear group is isomorphic to $G$:
\begin{equation*}
\left\{ \left[\begin{array}{cccccc}a_1&0&0&b_1&0&\mu_1'\\
0&a_2&0&0&b_2&\mu_2'\\
\lambda_1&\lambda_2&a_3&\mu_1&\mu_2&\kappa\\
c_1&0&0&d_1&0&-\lambda_1'\\
0&c_2&0&0&d_2&-\lambda_2'\\
0&0&0&0&0&a_3^{-1}
\end{array}\right]\in M_6(\mathbb{R}); \begin{array}{c}a_i,b_i,c_i,d_i,\lambda_i,\lambda_i',\mu_i,\mu_i',\kappa\in\mathbb{R},\\a_3\in\mathbb{R}_{>0},\quad a_id_i-b_ic_i=1,\\ \left[\lambda_i\,\,\mu_i\right]=a_3\left[\lambda_i'\,\,\mu_i'\right]\left[\begin{array}{cc}a_i&b_i\\c_i&d_i\end{array}\right]\\\qquad\qquad\qquad(i=1,2)\end{array}\right\}.
\end{equation*}
In more detail the linear group acts on $\mathcal{D}_5$ by linear fractional transformations, and the natural map from the linear group to $G$ gives rise to an isomorphism between the Lie groups.
\end{theorem}
Let us denote by $\matG$ the linear group given in Theorem \ref{Thefollowi}.
Let $E_1,E_2,E_3,E_{3,1},E_{3,2},A_1,A_2,A_3,A_{3,1},\allowbreak A_{3,2},W_1		$, and $W_2$ be the elements of $M_6(\mathbb{R})$ satisfying
\begin{equation*}\begin{split}
&e_1E_1+e_2E_2+e_3E_3+e_{3,1}E_{3,1}+e_{3,2}E_{3,2}\\&\quad+a_1A_1+a_2A_2+a_3A_3+a_{3,1}A_{3,1}+a_{3,2}A_{3,2}+k_1W_1+k_2W_2\\&=\left[\begin{array}{cccccc}\frac{a_1}{2}&0&0&e_1-k_1&0&e_{3,1}\\
0&\frac{a_2}{2}&0&0&e_2-k_2&e_{3,2}\\
a_{3,1}&a_{3,2}&\frac{a_3}{2}&e_{3,1}&e_{3,2}&e_3\\
k_1&0&0&-\frac{a_1}{2}&0&-a_{3,1}\\
0&k_2&0&0&-\frac{a_2}{2}&-a_{3,2}\\
0&0&0&0&0&-\frac{a_3}{2}
\end{array}\right]
\end{split}\end{equation*}
for $e_1, e_2, e_3, e_{3,1}, e_{3,2}, a_1, a_2, a_3, a_{3,1}, a_{3,2}, k_1, k_2\in\mathbb{R}$.
Then $\{E_1,E_2,E_3,E_{3,1},\allowbreak E_{3,2},A_1,A_2,A_3,A_{3,1},A_{3,2},W_1,W_2\}$ form a basis of $\mathfrak{g}'$, and we use the same symbols $E_1,E_2,E_3,\allowbreak E_{3,1},E_{3,2},A_1,A_2,A_3,A_{3,1},\allowbreak A_{3,2},W_1$, and $W_2$ for the corresponding elements of $\mathfrak{g}$.
Let $G_{iI_3}$ be the isotropy subgroup of $G$ at $iI_3\in\mathcal{D}_5$.
\begin{theorem}[\cite{Geatti87}]\label{Theidentit}
We have $G_{iI_3}=\exp \langle W_1, W_2\rangle$.
\end{theorem}
We have the following bracket relations:
\begin{equation*}\begin{split}
&[E_1,A_1]=-E_1,\\
&{[E_1,A_{3,1}]}=-E_{3,1},\\ 
&{[E_1,W_1]}=2A_1,\\
&{[E_2,A_2]}=-E_2,\\
&{[E_2,A_{3,2}]}=-E_{3,2},\\
&{[E_2,W_{2}]}=2A_2,\\
&{[E_3,A_3]}=-E_3,\end{split}\begin{split}
&\quad{[E_{3,1},A_1]}=-\tfrac{1}{2}E_{3,1},\\
&\quad{[E_{3,1},A_3]}=-\tfrac{1}{2}E_{3,1},\\
&\quad{[E_{3,1},A_{3,1}]}=-2E_3,\\
&\quad{[E_{3,1},W_1]}=A_{3,1},\\
&\quad{[E_{3,2},A_2]}=-\tfrac{1}{2}E_{3,2},\\
&\quad{[E_{3,2},A_3]}=-\tfrac{1}{2}E_{3,2},\\
&\quad{[E_{3,2},A_{3,2}]}=-2E_3,\\
&\quad{[E_{3,2},W_2]}=A_{3,2},\end{split}\begin{split}
&\quad{[A_1,A_{3,1}]}=-\tfrac{1}{2}A_{3,1},\\
&\quad{[A_1,W_1]}=-(W_1+2E_1),\\
&\quad{[A_2,A_{3,2}]}=-\tfrac{1}{2}A_{3,2},\\
&\quad{[A_2,W_2}]=-(W_2+2E_2),\\
&\quad{[A_3,A_{3,1}]}=\tfrac{1}{2}A_{3,1},\\
&\quad{[A_3,A_{3,2}]}=\tfrac{1}{2}A_{3,2},\\
&\quad{[A_{3,1},W_1]}=-E_{3,1},\\
&\quad{[A_{3,2},W_2]}=-E_{3,2}.
\end{split}\end{equation*}

\section{CS orbits of generic CS representations}\label{CSorbitand}

Let $M$ be a CS orbit of a generic CS representation $\pi$ of $G$, and let $K$ be the isotropy subgroup of $G$ at some point $m_0$ of $M$.
For a connected Riemannian manifold, every isotropy subgroup of the isometry group is compact.
Thus $\exp\mathrm{ad}_\mathfrak{g}\,\mathfrak{k}\subset\mathop{\mathrm{Int}}{\mathfrak{g}}$ is a compact subgroup, where for a Lie algebra $\arbitalgebra$, we denote by $\mathop{\mathrm{Int}}\arbitalgebra$ the subgroup $\exp \mathop{\mathrm{ad}}\arbitalgebra \subset GL(\arbitalgebra)$. 
It is known \cite{GPV68} that $G$ has trivial center and that $G_{iI_3}=\exp \langle W_1,W_2\rangle$ is a maximal compact subgroup of $G$.
Thus $\mathop{\mathrm{Int}}\mathfrak{g}$ is isomorphic to $G$. 
Moreover, any two maximal compact subgroups of $G$ are conjugate (see {\cite[Chapter 4, Theorem 3.5]{encyclopedia}}), so that we may and do assume that $\mathfrak{k}\subset\langle W_1,W_2\rangle$.
We then have $\mathfrak{k}=0$ or $\langle W_1,W_2\rangle$ because $M$ is an even-dimensional differentiable manifold.

We shall show that $\mathfrak{k}$ must equal $\langle W_1,W_2\rangle$. Arguing contradiction, assume that $\mathfrak{k}=0$.
Then $M$ is diffeomorphic to $G$.
We have the following theorem.
\begin{theorem}[{\cite[Chapter 4, Proposition 4.4 and Theorem 4.7]{encyclopedia}}]\label{Letarbitgr}
\begin{itemize}
\item[(a)] Let $G_0$ be a linear Lie group.
If $G_0$ equals $K_0D_0$ for some compact subgroup $K_0$ of $\arbitgroup$ and for some connected real split solvable Lie subgroup $D_0$ of $\arbitgroup$, then $K_0$ is a maximal compact subgroup of $\arbitgroup$.
\item[(b)]Let $\arbitgroup$ be a real algebraic linear group. Then the identity component of $\arbitgroup$ can be topologically decomposed into the direct product of the groups $K_0$ and $D_0$, where $K_0$ is a maximal compact subgroup of $G_0$ and $D_0$ a maximal real split solvable Lie subgroup of $\arbitgroup$. 
\end{itemize}
\end{theorem}
Thus it follows from Theorem \ref{Letarbitgr}(b) that $G$ is homeomorphic to $\mathcal{D}_5\times G_{iI_3}$.
Hence $\pi_1(G,e)=\pi_1(G_{iI_3},e)=\mathbb{Z}^2$, which contradicts that $M$ is simply connected.
Therefore we conclude that $\mathfrak{k}=\langle W_1,W_2\rangle$.

Now we have a $G$-equivariant diffeomorphism $\mathcal{D}_5\rightarrow M$.
Let us consider the K\"{a}hler structure $(j,g)$ on $\mathcal{D}_5$ which is the pullback, by the diffeomorphism, of the K\"{a}hler structure on $M$.
Also we can regard $\mathcal{D}_5$ as a K\"{a}hler manifold by means of the Bergman metric on $\mathcal{D}_5$.
Then it follows from {\cite[Theorem 6.1]{Dotti93}} that there exists a biholomorphism $\mathcal{D}_5\rightarrow M$ since $G$ acts on $(\mathcal{D}_5,j,g)$ by holomorphic isometries.
Thus the action of $G$ on $M$ induces an action of $G$ on $\mathcal{D}_5$ by holomorphic automorphisms, and the action is given by $G\times\mathcal{D}_5\ni (g,z)\mapsto \varphi(g)z\in\mathcal{D}_5$ for some automorphism $\varphi$ of $G$.
Let $\psi$ be the automorphism of $\mathfrak{g}$ satisfying $\psi^2=\mathrm{id}_\mathfrak{g}$, $\psi(E_1)=E_2$, $\psi(E_3)=E_3$, $\psi(E_{3,1})=E_{3,2}$, $\psi(A_1)=A_2$, $\psi(A_3)=A_3$, $\psi(A_{3,1})=A_{3,2}$, and $\psi(W_1)=W_2$, and let $\sigma$ be the automorphism of $\mathfrak{g}$ satisfying $\sigma(E_1)=-E_1$, $\sigma(E_2)=-E_2$, $\sigma(E_3)=-E_3$, $\sigma(E_{3,1})=-E_{3,1}$, $\sigma(E_{3,2})=-E_{3,2}$, $\sigma(A_1)=A_1$, $\sigma(A_2)=A_2$, $\sigma(A_3)=A_3$, $\sigma(A_{3,1})=A_{3,1}$, $\sigma(A_{3,2})=A_{3,2}$, $\sigma(W_1)=-W_1$, and $\sigma(W_2)=-W_2$.
For an automorphism $\varphi$ of $\mathfrak{g}$, let $\varphi^0=\mathrm{id}_\mathfrak{g}$, and let $\varphi^1=\varphi$.
\begin{proposition}\label{Anyautomor}
Any automorphism $\varphi$ of $\mathfrak{g}$ can be written as $\varphi=\psi^{\varepsilon}\circ \sigma^{\varepsilon'} \circ\Ad{g}$ for some $g\in G$ and for some $\varepsilon, \varepsilon'\in\{0,1\}$.
\end{proposition}
We postpone the proof to Section \ref{Irreducibl}.
The automorphisms $\psi$ and $\sigma$ lift to the automorphisms of $G$.
To simplify the notation, we use the same symbols $\psi$ and $\sigma$ for the lifts.
Then we see that $\psi$ induces a biholomorphism 
\begin{equation*}
\mathcal{D}_5\ni(z^1,z^2,z^3,z^4,z^5)\mapsto (z^2,z^1,z^3,z^5,z^4)\in \mathcal{D}_5
\end{equation*}
and $\sigma$ a biholomorphism
\begin{equation*}
\mathcal{D}_5\ni(z^1,z^2,z^3,z^4,z^5)\mapsto (-\overline{z^1},-\overline{z^2},-\overline{z^3},-\overline{z^4},-\overline{z^5})\in \overline{\mathcal{D}_5},
\end{equation*}
where $\overline{\mathcal{D}_5}$ denotes the conjugate manifold.
Thus by Proposition \ref{Letpimathc1}, $\pi$ or $\overline{\pi}$ is unitarily equivalent with $\pi_{\arbitline}$ for some $G$-equivariant holomorphic line bundle $\arbitline$ over $\mathcal{D}_5$.
Here for a subgroup $G_0\subset G$, by a $G_0$-equivariant bundle over $\mathcal{D}_5$, we mean a $G_0$-equivariant bundle over $\mathcal{D}_5$ such that the action of $G_0$ on the base space $\mathcal{D}_5$ is given by $G_0\times\mathcal{D}_5\ni (g,z)\mapsto gz\in\mathcal{D}_5$.
We note that $\mathcal{D}_5$ is a Stein manifold (see \cite{Chen}), and hence every holomorphic line bundle over $\mathcal{D}_5$ is trivial by the Oka-Grauert principle.
Hence the representation $\tau_{\arbitline}$ can be realized on the space $\mathcal{O}(\mathcal{D}_5)$ of holomorphic functions on $\mathcal{D}_5$.
We call such a representation of $G$ on $\mathcal{O}(\mathcal{D}_5)$ a {\it holomorphic multiplier representation of $G$ over $\mathcal{D}_5$}.
We get the following theorem.

\begin{theorem}\label{Letpimathc2}
Let $\pi$ be a generic CS representation of $G$. Then $\pi$ or $\overline{\pi}$ is unitarily equivalent with a unitarization of a holomorphic multiplier representation of $G$ over $\mathcal{D}_5$.
\end{theorem}

\section{Holomorphic multiplier representations over $\mathcal{D}_5$}\label{Holomorphi}

Let $\mathfrak{g}_-$ be the complex subalgebra of $\mathfrak{g}_\mathbb{C}$ given by 
\begin{equation*}
\mathfrak{g}_-=\left\{x+iy\in\mathfrak{g}_\mathbb{C};\dt e^{tx}iI_3+i\dt e^{ty}iI_3\in T_{iI_3}^{0,1}\mathcal{D}_5\right\},
\end{equation*}
where $T_{iI_3}^{0,1}\mathcal{D}_5$ denotes the antiholomorphic tangent vector space at $iI_3$.
By Tirao and Wolf \cite{TW70}, the isomorphism classes of $G$-equivariant holomorphic line bundles over $\mathcal{D}_5$ stand in one-one correspondence with the one-dimensional complex representations of $\mathfrak{g}_-$ whose restrictions to $\mathfrak{g}_{iI_3}$ lift to representations of $G_{iI_3}$.
For a basis $\{x_\lambda\}$ of $\mathfrak{g}$, we shall denote the dual basis by $\{x_\lambda^*\}$. Let $\mathcal{M}$ be the set consists of all linear forms $\xi$ on $\mathfrak{g}$ given by 
\begin{equation*}\begin{split}
\xi=\xi(\xi_3, \eta_3, n, n')=\xi_3E_{3}^*+\eta_3A_{3}^*+\frac{n}{2}(2W_\one^*-E_{1}^*)+\frac{n'}{2}(2W_2^*-E_{2}^*),
\end{split}\end{equation*}
with $\xi_3, \eta_3\in\mathbb{R}$ and $n, n' \in\mathbb{Z}$.
If $\xi$ is extended to a complex linear form on $\mathfrak{g}_\mathbb{C}$, then $i\xi|_{\mathfrak{g}_-}\,(\xi\in\mathcal{M})$ defines a one-dimensional complex representation of $\mathfrak{g}_-$ whose restriction to $\mathfrak{g}_{iI_3}$ lifts to a representation of $G_{iI_3}$.
For $\xi\in\mathcal{M}$, let $\arbitline$ be a $G$-equivariant holomorphic line bundle over $\mathcal{D}_5$ whose isomorphism class corresponds to $i\xi|_{\mathfrak{g}_-}$, and put $\tau_\xi=\tau_{\arbitline}$. Also we put $\pi_\xi=\pi_{\arbitline}$ when $\tau_{\arbitline}$ is unitarizable.
Let
\begin{equation}\label{G1xixynnx0}\begin{split}
\Theta^G(n,n')=\{\xi(\xi_3, \eta_3, n, n'); \xi_3<0, \eta_3\in\mathbb{R}\}\quad(n, n'\in\mathbb{Z}_{>0}),\\
\Theta^G(\eta_3,n,n') =\{\xi(0,\eta_3,n,n')\}\quad(\eta_3\in\mathbb{R},n,n'\in\mathbb{Z}_{\geq 0}).
\end{split}\end{equation}
Then we have the following theorem.
\begin{theorem}[{\cite{Arashi20}}, {\cite[Theorem 13(i) and (iii)]{Ishi11}}]\label{beginitemi}
\begin{itemize}
 \item[(a)] For $\xi\in\mathcal{M}$, the representation $\tau_\xi$ is unitarizable if and only if $\xi$ belongs to any of the sets in \eqref{G1xixynnx0}.
 \item[(b)]
 For $\xi,\xi'\in\mathcal{M}$ with $\tau_\xi$, $\tau_{\xi'}$ unitarizable, the representations $\pi_\xi$ and $\pi_{\xi'}$ are unitarily equivalent if and only if $\xi$ and $\xi'$ belongs to the same set in \eqref{G1xixynnx0}.
\item[(c)]
Every holomorphic multiplier representation of $G$ over $\mathcal{D}_5$ is unitarily equivalent with $\pi_\xi$ for some $\xi\in\mathcal{M}$.
\end{itemize}
\end{theorem}

From now on, for $\xi\in\mathcal{M}$ such that $\tau_\xi$ is unitarizable, we think of $\pi_\xi$ as any of the holomorphic multiplier representations of $G$ over $\mathcal{D}_5$.
We shall mention the converse of Theorem \ref{Letpimathc2}.
Let $\mathcal{H}^\xi$ be the representation space of $\pi_\xi$, let $\mathcal{K
}^\xi:\mathcal{D}_5\times\mathcal{D}_5\rightarrow\mathbb{C}$ be the reproducing kernel of $\mathcal{H
}^\xi$, and let $\mathcal{K}_{iI_3}^\xi\in\mathcal{H}^\xi$ be the function given by $\mathcal{K}_{iI_3}^\xi(z)=\mathcal{K}^\xi(z,iI_3)\,(z\in\mathcal{D}_5)$.
If the representation $d\pi_\xi$ of $\mathfrak{g}$ is extended to a complex representation, then we have 
\begin{equation}\label{dpixiover}
d\pi_\xi(\overline{x})\mathcal{K}_{iI_3}^\xi=i\overline{\xi(x)}\mathcal{K}_{iI_3}^\xi\quad(x\in\mathfrak{g}_-),
\end{equation}
which implies that $\pi_\xi$ is an irreducible CS representation of $G$ if $\dim \mathcal{H}^\xi>1$ (see {\cite[Proposition 4.1]{Lisiecki95}}).

\section{Generic CS representations}\label{GenericCSr}

For $n,n'\in\mathbb{Z}_{> 0}$, let $\xi_{n,n'}$ be any of the elements of $\Theta^G(n,n')$.
\begin{proposition}\label{Foranynnin}
For any $n,n',l,l'\in\mathbb{Z}_{>0}$, the representations $\pi_{\xi_{n,n'}}$ and $\overline{\pi_{\xi_{l,l'}}}$ are not unitarily equivalent.
\end{proposition}
\begin{proof}
Let $\mathfrak{b}=\langle E_1, E_2, E_3, E_{3,1}, E_{3,2}, A_1, A_2, A_3, A_{3,1}, A_{3,2}\rangle$, and let $B=\exp\mathfrak{b}\subset G$.
It is enough to show that $\pi_{\xi_{n,n'}}$ and $\overline{\pi_{\xi_{l,l'}}}$ are not equivalent as unitary representations of $B$.
Note that $B$ is an exponential solvable Lie group, so that the equivalence classes of irreducible unitary representations of $B$ are in one-one correspondence with the coadjoint orbits of $B$ in $\mathfrak{b}^*$ (see \cite{Bernat}).
By {\cite[Theorem 13(ii)]{Ishi11}}, the equivalence classes of $\pi_{\xi_{n,n'}}|_B$ and $\pi_{\xi_{l,l'}}|_B$ correspond to the coadjoint orbit through $-(E_1^*+E_2^*+E_3^*)|_\mathfrak{b}\in\mathfrak{b}^*$ (see Remark \ref{Forunderli} below for more detail).
Then we see that the equivalence class of $\overline{\pi_{\xi_{l,l'}}}$ corresponds to the coadjoint orbit through $(E_1^*+E_2^*+E_3^*)|_\mathfrak{b}\in\mathfrak{b}^*$.

Let $\eta$ be a linear form on $\mathfrak{b}$, and suppose that $\langle E_3,\eta\rangle >0$.
We have $\Ad{e^{tA_3}} E_3=e^tE_3\,(t\in\mathbb{R})$, and $E_3$ commutes with $E_1, E_2, E_3, E_{3,1}, E_{3,2}, A_1, A_2, A_{3,1}$, and $A_{3,2}$.
Thus $\langle E_3,\cAd{b}\eta \rangle=\langle \Ad{b^{-1}} E_3,\eta\rangle>0$ for $b\in B$.
This implies that the coadjoint orbit through $-(E_1^*+E_2^*+E_3^*)|_\mathfrak{b}$ and the one through $(E_1^*+E_2^*+E_3^*)|_\mathfrak{b}$ are different.
The proof is complete.
\end{proof}
\begin{remark}\label{Forunderli}
For $\xi=\xi(\xi_3,\eta_3,n,n')\in\Theta^G(n,n')$, we shall show that the equivalence class of $\pi_{\xi}|_B$ corresponds to the coadjoint orbit through $-(E_1^*+E_2^*+E_3^*)|_{\mathfrak{b}}$.
For $\underline{s}=(s_1, s_2, s_3)\in\mathbb{C}^3$, let $\alpha^{\underline{s}}=\sum_{k=1}^3 s_k A_k^*|_\mathfrak{b}\in(\mathfrak{b}^*)_\mathbb{C}$, and let $\chi^{\underline{s}}$ be the character of $B$ given by $\chi^{\underline{s}}(\exp x)=\exp\alpha^{\underline{s}}(x)\,(x\in\mathfrak{b})$.
Let us consider the action of $B$ on the holomorphic line bundle $\mathcal{D}_5\times \mathbb{C}$ given by 
\begin{equation*}
B\times\mathcal{D}_5\times\mathbb{C}\ni (b,z,\zeta)\mapsto (bz,\chi^{-\underline{s}/2}(b)\zeta)\in\mathcal{D}_5\times\mathbb{C},
\end{equation*}
and we denote the $B$-equivariant holomorphic line bundle by $L^{\underline{s}}$.
Now the isomorphism classes of $B$-equivariant holomorphic line bundles over $\mathcal{D}_5$ stand in one-one correspondence with the one-dimensional complex representations of $\mathfrak{b}_\mathbb{C}\cap\mathfrak{g}_-$, and $L^{\underline{s}}$ corresponds to $-\tfrac{i}{2}(\sum_{k=1}^3\mathrm{Re}\,s_kE_k^*+\mathrm{Im}\,s_kA_k^*)|_{\mathfrak{b}_\mathbb{C}\cap\mathfrak{g}_-}$, where we extend $\sum_{k=1}^3\mathrm{Re}\,s_kE_k^*+\mathrm{Im}\,s_kA_k^*$ to a complex linear form on $\mathfrak{g}_\mathbb{C}$.
Hence for $\underline{s}=(n,n',-2(\xi_3+i\eta_3))$, the representation $\pi_{L^{\underline{s}}}$ of $B$ is unitarily equivalent with $\pi_\xi|_B$.
For $\alpha\in\mathfrak{g}^*$, let $\mathfrak{b}_{\alpha}=\{ x\in \mathfrak{b};[y,x]=\alpha(y)x\quad\mbox{for all}\,y\in\langle A_1, A_2, A_3\rangle\}$.
Put $q_k=\sum_{3 \geq l>k\geq 1}\dim \mathfrak{b}_{(A_l^*-A_k^*)/2}\,(k=1,2,3)$.
Then we have $q_1=\dim\langle A_{3,1}\rangle=1$, $q_2=\dim \langle A_{3,2}\rangle =1$, $q_3=0$, and hence
\begin{equation}\label{Reskqk2qua}
\mathrm{Re}\,s_k>q_k/2\quad(k=1,2,3).
\end{equation}
According to {\cite[Theorem 13(ii)]{Ishi11}}, we can obtain the desired result from \eqref{Reskqk2qua}.
\end{remark}
Let us consider the set of equivalence classes of irreducible unitary representations of $G$.
For a unitary representation $\pi_0$ of a Lie group $G_0$, we denote the equivalence class of $\pi_0$ by $[\pi_0]$.
\begin{theorem}\label{Thesetofeq}
The set of equivalence classes of generic CS representations of $G$ is given by 
\begin{equation*}
\{[\pi_{\xi_{n,n'}}];n,n'\in\mathbb{Z}_{>0}\}\sqcup\{[\overline{\pi_{\xi_{n,n'}}}];n,n'\in\mathbb{Z}_{>0}\}.
\end{equation*}
\end{theorem}
\begin{proof}
By Theorems \ref{Letpimathc2} and \ref{beginitemi}, it is enough to show that
\begin{itemize}
 \item[(a)] For any $n,n'\in\mathbb{Z}_{>0}$, and $\xi\in \Theta^G(n,n')$, the representation $\pi_\xi$ is generic,
 \item[(b)] For any $\eta\in\mathbb{R}, n,n'\in\mathbb{Z}_{\geq 0}$, and $\xi\in\Theta^G(\eta_3,n,n')$, the representation $\pi_\xi$ is not generic.
\end{itemize}
For $\xi\in\mathcal{M}$ with $\tau_\xi$ unitarizable, we have $\mu_{\pi_\xi}([\mathcal{K}_{iI_3}^\xi])=\xi$, and hence we can identify the coadjoint orbit through $\xi\in\mathfrak{g}^*$ with the CS orbit through $[\mathcal{K}_{iI_3}^\xi]\in\mathbb{P}(\mathcal{H}^\xi)$.
We denote by $\alpha$ the action of $G$ on the coadjoint orbit through $\xi$.
Let $G_\xi$ be the isotropy subgroup of $G$ at $\xi$. We note that
$\mathfrak{g}_\xi=\{x\in\mathfrak{g};\xi([x,y])=0\text{ for all }y\in\mathfrak{g}\}$.
The matrix of the skew-symmetric bilinear form $\xi([x,y])$ with respect to the basis $\{E_1,E_2,E_3,E_{3,1},E_{3,2},A_1,A_2,A_3,A_{3,1},A_{3,2},W_1,W_2\}$ is given by 
\begin{equation*}
\left[\begin{array}{cccccccccccc}0 & 0 & 0 & 0 & 0 & \frac{n}{2} & 0 & 0 & 0 & 0 & 0 & 0\\
0 & 0 & 0 & 0 & 0 & 0 & \frac{n'}{2} & 0 & 0 & 0 & 0 & 0\\
0 & 0 & 0 & 0 & 0 & 0 & 0 & -\xi_3 & 0 & 0 & 0 & 0\\
0 & 0 & 0 & 0 & 0 & 0 & 0 & 0 & -2 \xi_3 & 0 & 0 & 0\\
0 & 0 & 0 & 0 & 0 & 0 & 0 & 0 & 0 & -2 \xi_3 & 0 & 0\\
-\frac{n}{2} & 0 & 0 & 0 & 0 & 0 & 0 & 0 & 0 & 0 & 0 & 0\\
0 & -\frac{n'}{2} & 0 & 0 & 0 & 0 & 0 & 0 & 0 & 0 & 0 & 0\\
0 & 0 & \xi_3 & 0 & 0 & 0 & 0 & 0 & 0 & 0 & 0 & 0\\
0 & 0 & 0 & 2 \xi_3 & 0 & 0 & 0 & 0 & 0 & 0 & 0 & 0\\
0 & 0 & 0 & 0 & 2 \xi_3 & 0 & 0 & 0 & 0 & 0 & 0 & 0\\
0 & 0 & 0 & 0 & 0 & 0 & 0 & 0 & 0 & 0 & 0 & 0\\
0 & 0 & 0 & 0 & 0 & 0 & 0 & 0 & 0 & 0 & 0 & 0\end{array}\right].
\end{equation*}

(a) Since $\xi_3<0$ and $n,n'\in\mathbb{Z}_{>0}$, it follows that $\mathfrak{g}_\xi=\langle W_1,W_2\rangle$.
Now we have $\mathrm{Lie}(\ker \pi_\xi)\subset \mathrm{Lie}(\ker \alpha)=0$, and hence $\pi_\xi$ is generic.

(b) We have $E_3\in\mathfrak{g}_\xi$.
Thus $\dim\ker\alpha\geq 1$.
We see from \eqref{mathrmLiek} that $\pi_\xi$ is not generic.
\end{proof}

\section{Irreducible non-generic CS representations}\label{Irreducibl}

Let $\mathfrak{h}_5=\langle E_3,E_{3,1},E_{3,2}, A_{3,1},A_{3,2}\rangle$, $\mathfrak{h}_3=\langle E_3,E_{3,1},A_{3,1}\rangle$, $\mathfrak{h}_3'=\langle E_3,E_{3,2},A_{3,2}\rangle$, $\mathfrak{a}_1=\langle A_3\rangle$, $\mathfrak{s}_3=\langle E_1,A_1,W_1\rangle$, $\mathfrak{s}_3'=\langle E_2,A_2,W_2\rangle$.
Then we have the following lemma.
\begin{lemma}\label{l}
\begin{itemize}
\item[(a)] Every nontrivial ideal in $\mathfrak{g}$ contains $\langle E_3\rangle$.
\item [(b)]Let $\mathfrak{h}$ be an ideal in $\mathfrak{g}$ such that $\langle E_3\rangle \subsetneq \mathfrak{h}$. Then $\mathfrak{h}$ contains $\mathfrak{h}_3$ or $\mathfrak{h}'_3$.
\item[(c)] Let $\mathfrak{h}$ be an ideal in $\mathfrak{g}$ such that $\mathfrak{h}'_3\subsetneq\mathfrak{h}$. Then $\mathfrak{h}$ contains $\mathfrak{h}_5$ or $\mathfrak{h}'_3\oplus\mathfrak{s}'_3$.
\item[(d)] Let $\mathfrak{h}$ be an ideal in $\mathfrak{g}$ such that $\mathfrak{h}'_3\oplus\mathfrak{s}'_3\subsetneq\mathfrak{h}$. Then $\mathfrak{h}$ contains $\mathfrak{h}_5\oplus\mathfrak{s}'_3$. 
\end{itemize}
\end{lemma}
\begin{proof}
Let $x=e_1E_1+e_2E_2+e_3E_3+e_{3,1}E_{3,1}+e_{3,2}E_{3,2}+a_1A_1+a_2A_2+a_3A_3+a_{3,1}A_{3,1}+a_{3,2}A_{3,2}+k_1W_1+k_2W_2$ with $e_1, e_2, e_3, e_{3,1}, e_{3,2}, a_1, a_2, a_3, a_{3,1}, a_{3,2}, k_1, k_2\in\mathbb{R}$.

(a) Suppose that $x$ is contained in an ideal $\mathfrak{h}$ of $\mathfrak{g}$ such that $E_3\notin\mathfrak{h}$.
Since $[E_3,x]=-a_3E_3$, we have $a_3=0$. Then we have
\begin{equation*}\begin{split}
[E_{3,1},x]&=-\tfrac{a_1}{2}E_{3,1}-2a_{3,1}E_3+k_1A_{3,1},\\
[E_{3,1},[E_{3,1},x]]&=-2k_1E_3,\quad \quad[A_{3,1},[E_{3,1},x]]=-a_1E_3,\\
[E_{3,2},x]&=-\tfrac{a_2}{2}E_{3,2}-2a_{3,2}E_3+k_2A_{3,2},\\
[E_{3,2},[E_{3,2},x]]&=-2k_2E_3,\quad\quad[A_{3,2},[E_{3,2},x]]=-a_2E_3,
\end{split}\end{equation*}
so that $a_1=a_2=a_{3,1}=a_{3,2}=k_1=k_2=0$. Next,
\begin{equation*}\begin{split}
[A_{3,1},x]&=e_1E_{3,1}+2e_{3,1}E_3,\quad\quad[A_{3,1},[A_{3,1},x]]=2e_1E_3,\\
[A_{3,2},x]&=e_2E_{3,2}+2e_{3,2}E_3,\quad\quad[A_{3,2},[A_{3,2},x]]=2e_2E_3,
\end{split}\end{equation*}
which imply that $e_1=e_2=e_{3,1}=e_{3,2}=0$.
Thus $x=e_3 E_{3}=0$ and $\mathfrak{h}=0$. 
Therefore, every nontrivial ideal of $\mathfrak{g}$ contains $E_3$. 

(b)
Let $\mathfrak{h}'=\langle E_3\rangle$.
It is enough to show that $\tilde{\mathfrak{h}}=\mathfrak{h}/\mathfrak{h}'$ contains $sE_{3,1}+tA_{3,1}+\mathfrak{h}'$ with $s^2+t^2\neq 0$ or $sE_{3,2}+tA_{3,2}+\mathfrak{h}'$ with $s^2+t^2\neq 0$.
Arguing contradiction, assume that $\tilde{\mathfrak{h}}$ does not contain either of them.
Let $x\in\mathfrak{h}$.
We have
\begin{equation}\begin{split}\label{E31xfrac12}
[E_{3,1},x]&=-\tfrac{1}{2}(a_1+a_3)E_{3,1}+k_1A_{3,1},\quad
[E_{3,2},x]=-\tfrac{1}{2}(a_2+a_3)E_{3,2}+k_2A_{3,2},\\
[A_3,x]&=\tfrac{1}{2}(e_{3,1}E_{3,1}+e_{3,2}E_{3,2}+a_{3,1}A_{3,1}+a_{3,2}A_{3,2}),\\
[W_1,[A_3,x]]&=\tfrac{1}{2}(a_{3,1}E_{3,1}-e_{3,1}A_{3,1})\quad(\mathrm{mod}\,\mathfrak{h}'),
\end{split}\end{equation}
so that $e_{3,1}=e_{3,2}=a_1+a_3=a_2+a_3=a_{3,1}=a_{3,2}=k_1=k_2=0$.
We also have
\begin{equation*}
[A_{3,1},x]=e_1E_{3,1}+\tfrac{a_1-a_3}{2}A_{3,1},\quad
[A_{3,2},x]=e_2E_{3,2}+\tfrac{a_2-a_3}{2}A_{3,2}\quad(\mathrm{mod}\,\mathfrak{h}'),
\end{equation*}
so that $e_1=e_2=a_1=a_2=a_3=0$. Hence, $\tilde{\mathfrak{h}}=0$, which contradicts the assumption.

(c)
Let $\mathfrak{h}'=\mathfrak{h}'_3$.
It is enough to show that $\tilde{\mathfrak{h}}=\mathfrak{h}/\mathfrak{h}'$ contains $sE_{3,1}+tA_{3,1}+\mathfrak{h}'$ with $s^2+t^2\neq 1$ or $sE_2+tA_2+\mathfrak{h}'$ with $s^2+t^2\neq 0$. 
Arguing contradiction, assume that $\tilde{\mathfrak{h}}$ does not contain either of them.
Let $x\in\mathfrak{h}$.
We have 
\begin{equation}\begin{split}\label{E1a1E1a312}
[E_1,x]&=-a_1E_1-a_{3,1}E_{3,1}+2k_1A_1,\quad[A_{3,1},[E_1,x]]=-a_1E_{3,1}+k_1A_{3,1},\\
[A_3,x]&=\frac{1}{2}(e_{3,1}E_{3,1}+a_{3,1}A_{3,1}),\quad[E_2,x]=-a_2E_2+2k_2A_2\quad(\mathrm{mod}\,\mathfrak{h}'),
\end{split}\end{equation}
so that $e_{3,1}=a_1=a_2=a_{3,1}=k_1=k_2=0$.
Next,
\begin{equation*}[A_2,x]=e_2E_2,\quad[A_{3,1},x]=e_1E_{3,1}-\tfrac{a_3}{2}A_{3,1}\quad(\mathrm{mod}\,\mathfrak{h}'),
\end{equation*}
which implies that $e_1=e_2=a_3=0$.
Hence, $x\in\mathfrak{h}'$, which contradicts the assumption.

(d)
Let $\mathfrak{h}'=\mathfrak{h}'_3\oplus\mathfrak{s}'_3$.
It is enough to show that $\tilde{\mathfrak{h}}=\mathfrak{h}/\mathfrak{h}'$ contains $sE_{3,1}+tA_{3,1}+\mathfrak{h}'$ with $s^2+t^2\neq 0$.
Arguing contradiction, assume that $\tilde{\mathfrak{h}}$ does not contain such an element.
Let $x\in\mathfrak{h}$.
From \eqref{E31xfrac12} and \eqref{E1a1E1a312}, we see that $e_{3,1}=a_1=a_3=a_{3,1}=k_1=0$. Since 
\begin{equation*}\begin{split}
[W_1,x]&=-2e_1A_1, \quad [A_{3,1},[W_1,x]]=-e_1A_{3,1}\quad(\mathrm{mod}\,\mathfrak{h}'),
\end{split}\end{equation*}
we obtain $e_1=0$.
Hence, $x\in\mathfrak{h}'$, which contradicts the assumption.
\end{proof}

If we take into account Lemma \ref{l} and that $\mathfrak{s}_3\oplus\mathfrak{s}_3'$ is semisimple, it is not hard to determine all ideals of $\mathfrak{g}$.
Figure \ref{fig1} gives the Hasse diagram of the set of all nontrivial ideals of $\mathfrak{g}$, ordered by inclusion.
\begin{proof}[Proof of Proposition \textup{\ref{Anyautomor}}]
We have $\varphi (\mathfrak{h}_3)=\mathfrak{h}_3$ or $\varphi(\mathfrak{h}_3)=\mathfrak{h}'_3$, and it
is enough to show that if $\varphi (\mathfrak{h}_3)=\mathfrak{h}_3$, then $\varphi$ can be written as $\varphi=\sigma^{\varepsilon'} \circ\Ad{g}$ for some $g\in G$ and for some $\varepsilon'\in\{0,1\}$.
Let $\varphi(\mathfrak{h}_3)=\mathfrak{h}_3$.
Let us consider the adjoint action of $G$ on $\mathfrak{g}$.
The subgroups $\exp \mathfrak{a}_1, \exp \mathfrak{s}_3\subset G$ act on the ideal $\mathfrak{h}_3$ of $\mathfrak{g}$ by dilations 
\begin{equation*}
\mathfrak{h}_3\ni e_3E_3+e_{3,1}E_{3,1}+a_{3,1}A_{3,1}\mapsto r^2e_3E_3+re_{3,1}E_{3,1}+ra_{3,1}A_{3,1}\in
\mathfrak{h}_3\quad(r>0)
\end{equation*}
and symplectic maps
\begin{equation*}\begin{split}
\mathfrak{h}_3\ni e_3E_3+e_{3,1}E_{3,1}+a_{3,1}A_{3,1}\mapsto e_3E_3+(e_{3,1}\alpha +a_{3,1}\beta)E_{3,1}+(e_{3,1}\gamma+&a_{3,1}\delta)A_{3,1}\in\mathfrak{h}_3\\&(\alpha\delta-\beta\gamma=1),
\end{split}\end{equation*}
respectively.
It is well known that the automorphism group of $\mathfrak{h}_3$ is generated by inner automorphisms, symplectic maps, dilations, and inversion
\begin{equation*}
\mathfrak{h}_3\ni e_3E_3+e_{3,1}E_{3,1}+a_{3,1}A_{3,1}\mapsto -e_3E_3+a_{3,1}E_{3,1}+e_{3,1}A_{3,1}\in\mathfrak{h}_3.
\end{equation*}
Thus we have $\varphi\circ \Ad{g}|_{\mathfrak{h}_3}=\mathrm{id}_{\mathfrak{h}_3}$ or $\varphi\circ \Ad{g}|_{\mathfrak{h}_3}=\sigma|_{\mathfrak{h}_3}$ for some $g\in\exp\mathfrak{h}_3\oplus\mathfrak{a}_1\oplus\mathfrak{s}_3\subset G$.

Now it is enough to show that if $\varphi|_{\mathfrak{h}_3}=\mathrm{id}_{\mathfrak{h}_3}$, then $\varphi\circ \Ad{g}=\mathrm{id}_{\mathfrak{g}}$ for some $g\in G$.
Let $\varphi|_{\mathfrak{h}_3}=\mathrm{id}_{\mathfrak{h}_3}$.
We have $\varphi\circ \Ad{g}|_{\mathfrak{h}_5}=\mathrm{id}_{\mathfrak{h}_5}$ for some $g\in G$.
Hence we may and do assume that $\varphi|_{\mathfrak{h}_5}=\mathrm{id}_{\mathfrak{h}_5}$.
Let us consider the subrepresentation $(\mathrm{ad},\mathfrak{h}_5)$ of the adjoint representation $\mathrm{ad}$ of $\mathfrak{g}$.
Then the kernel of the subrepresentation equals $\langle E_3\rangle$ (see Remark \ref{Forxe1E1e2} below), and hence it follows that $\varphi(A_3)=A_3+e_3E_3\,$ with $e_3\in\mathbb{R}$.
Moreover we see that $\varphi\circ \Ad{g}|_{\mathfrak{h}_5\oplus\mathfrak{a}_1}=\mathrm{id}_{\mathfrak{h}_5\oplus\mathfrak{a}_1}$ for some $g\in \exp\langle E_3\rangle\subset G$.
Let us consider the subrepresentation $(\mathrm{ad},\mathfrak{h}_5\oplus\mathfrak{a}_1)$ of the adjoint representation $\mathrm{ad}$ of $\mathfrak{g}$.
Then the kernel of the subrepresentation equals $\{0\}$ (see Remark \ref{Forxe1E1e2} below), and hence it follows that an automorphism of $\mathfrak{g}$ which is the identity on $\mathfrak{h}_5\oplus\mathfrak{a}_1$ is the identity on $\mathfrak{g}$.
The proof is complete.
\end{proof}
\begin{remark}\label{Forxe1E1e2}
For $x=e_1E_1+e_2E_2+e_3E_3+e_{3,1}E_{3,1}+e_{3,2}E_{3,2}+a_1A_1+a_2A_2+a_3A_3+a_{3,1}A_{3,1}+a_{3,2}A_{3,2}+k_1W_1+k_2W_2$ with $e_1, e_2, e_3, e_{3,1}, e_{3,2}, a_1, a_2, a_3, a_{3,1}, a_{3,2}, k_1, k_2\in\mathbb{R}$, let us see the matrix of $\ad{x}:\mathfrak{h}_5\oplus\mathfrak{a}_1\rightarrow\mathfrak{h}_5\oplus\mathfrak{a}_1$ with respect to the basis $\{E_3,E_{3,1},E_{3,2},A_3,A_{3,1},A_{3,2}\}$.
The matrix is given by 
\begin{equation*}
\left[\begin{array}{cccccc}
		a_{3}&	2a_{3,1}&	2a_{3,2}&	-e_{3}&	-2e_{3,1}&	-2e_{3,2}\\
		0&	a_{3}/2+a_{1}/2&	0&	-e_{3,1}/2&	k_{1}-e_{1}&	0\\
		0&	0&	a_{3}/2+a_{2}/2&	-e_{3,2}/2&	0&	k_{2}-e_{2}\\
		0&	0&	0&	0&	0&	0\\
		0&	-k_{1}&	0&	-a_{3,1}/2&	a_{3}/2-a_{1}/2&	0\\
		0&	0&	-k_{2}&	-a_{3,2}/2&	0&	a_{3}/2-a_{2}/2
\end{array}\right].
\end{equation*}
\end{remark}
By the definition of CS representation, if all generic 
CS representations of all the quotient groups of $G$ by connected closed normal subgroups are given, then we can obtain all irreducible CS representations of $G$ by composing the quotient maps.
Let $\tilde{G}$ be the quotient group by a connected closed normal subgroup of $G$, and let $\tilde{G}\neq G, \{e\}$.
According to Figure \ref{fig1}, it is enough to consider the following cases:
\begin{equation*}\begin{split}
&\textup{(i)}\,\tilde{\mathfrak{g}}\simeq\mathbb{R},\quad \textup{(ii)}\,\tilde{\mathfrak{g}}\simeq\mathfrak{sl}(2,\mathbb{R}),\quad
\textup{(iii)}\,\tilde{\mathfrak{g}}\simeq\mathbb{R}\oplus\mathfrak{sl}(2,\mathbb{R}),\quad \textup{(iv)}\,\tilde{\mathfrak{g}}\simeq\mathfrak{sl}(2,\mathbb{R})\oplus\mathfrak{sl}(2,\mathbb{R}),
\\&\textup{(v)}\,\tilde{\mathfrak{g}}\simeq\mathbb{R}\oplus\mathfrak{sl}(2,\mathbb{R})\oplus\mathfrak{sl}(2,\mathbb{R}),\quad \textup{(vi)}\,\tilde{\mathfrak{g}}=\mathfrak{g}/\mathfrak{h}_3'\oplus\mathfrak{s}_3',\quad\textup{(vii)}\,\tilde{\mathfrak{g}}=\mathfrak{g}/\mathfrak{h}_3',\quad\textup{(viii)}\,\tilde{\mathfrak{g}}=\mathfrak{g}/\langle E_3\rangle.
\end{split}\end{equation*}
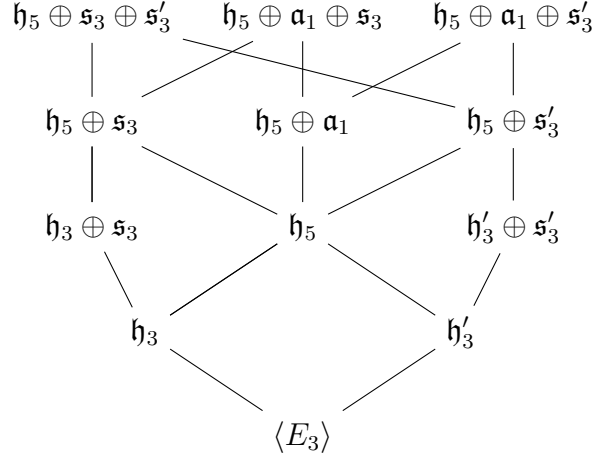
\begin{figure}[htpp]
\begin{tikzpicture}[scale=.7]
 
 \node (-40) at (-4,0) {$\mathfrak{h}_5\oplus\mathfrak{s}_3\oplus\mathfrak{s}_3'$};
 \node (00) at (0,0) {$\mathfrak{h}_5\oplus\mathfrak{a}_1\oplus\mathfrak{s}_3$};
 \node (40) at (4,0) {$\mathfrak{h}_5\oplus\mathfrak{a}_1\oplus\mathfrak{s}_3'$};
 \node (-4-2) at (-4,-2) {$\mathfrak{h}_5\oplus\mathfrak{s}_3$};
 \node (0-2) at (0,-2) {$\mathfrak{h}_5\oplus\mathfrak{a}_1$};
 \node (4-2) at (4,-2) {$\mathfrak{h}_5\oplus\mathfrak{s}_3'$};
 \node (-4-4) at (-4,-4) {$\mathfrak{h}_3\oplus\mathfrak{s}_3$};
 \node (0-4) at (0,-4)
 {$\mathfrak{h}_5$};
 \node (4-4) at (4,-4) {$\mathfrak{h}_3'\oplus\mathfrak{s}_3'$};
 \node (-3-6) at (-3,-6) {$\mathfrak{h}_3$};
 \node (3-6) at (3,-6) {$\mathfrak{h}_3'$};
 \node (0-8) at (0,-8) {$\langle E_3\rangle$};
 \draw (-40) -- (4-2);
 \draw (00) -- (-4-2) -- (-4-4);
 \draw (00) --(0-2) -- (0-4) -- (3-6);
 \draw (40) -- (0-2);
\draw (40) -- (4-2) -- (0-4) -- (-3-6);
 \draw (-40) -- (-4-2) -- (0-4) -- (-3-6);
 \draw (-4-2) -- (-4-4) -- (-3-6) -- (0-8);
 \draw (4-2) -- (4-4) -- (3-6) -- (0-8);
\end{tikzpicture} \caption{The Hasse diagram of the set of all nontrivial ideals of $\mathfrak{g}$}\label{fig1}
\end{figure}

However $\tilde{G}$ does not admit generic CS representations in the cases (vi)-(viii).
We shall prove this.
Suppose that $M$ is a CS orbit of a generic CS representation of $\tilde{G}$.
Let $K$ be the isotropy subgroup of $\tilde{G}$ at some point $m_0$ of $M$.
We shall seek a maximal compact subgroup of $\mathop{\mathrm{Int}}\tilde{\mathfrak{g}}$.
\begin{proposition}
\begin{itemize}
\item[(a)] We have the following isomorphisms:
$\Int{\tilde{\mathfrak{g}}}\simeq G/\exp \mathfrak{h}_3'\oplus\mathfrak{s}_3'$ in the case \textup{(vi)}, $\Int{\tilde{\mathfrak{g}}}\simeq G/\exp \mathfrak{h}'_3$ in the case \textup{(vii)}, and $\Int{\tilde{\mathfrak{g}}}\simeq G/\exp \langle E_3\rangle$ in the case \textup{(viii)}.
\item[(b)]The maximal compact subgroups of $G/\exp \mathfrak{h}_3'\oplus\mathfrak{s}_3'$, $ G/\exp \mathfrak{h}'_3$, and $G/\exp \langle E_3\rangle$ are conjugate to the images of $\exp \langle W_1, W_2\rangle\subset G$ under the quotient maps.
\item[(c)] We have $\pi_1(\Int{\tilde{\mathfrak{g}}},e)=\mathbb{Z}$ in the case \textup{(vi)} and $\pi_1(\Int{\tilde{\mathfrak{g}}},e)=\mathbb{Z}^2$ in the cases \textup{(vii)} and \textup{(viii)}.
\end{itemize}
\end{proposition}
\begin{proof}
We shall prove (a) and (b) for the case (viii) and (c) for the case (vi).
For the other cases, this can be proved in the same way.

(a) It is enough to show that $G/\exp\langle E_3\rangle$ has trivial center.
Let 
\begin{equation*}
g=\left[\begin{array}{cccccc}a_1&	0&	0&	b_1&	0&	\mu'_1\\
		0&	a_2&	0&	0&	b_2&	\mu'_2\\
		(c_1\mu'_1+a_1\lambda'_1)a_3&	(c_2\mu'_2+a_2\lambda'_2)a_3&	a_3&	(d_1\mu'_1+b_1\lambda'_1)a_3&	(d_2\mu'_2+b_2\lambda'_2)a_3&	\kappa\\
		c_1&	0&	0&	d_1&	0&	-\lambda'_1\\
		0&	c_2&	0&	0&	d_2&	-\lambda'_2\\
		0&	0&	0&	0&	0&	1/a_3
\end{array}\right]\in\matG.
\end{equation*}
Then
\begin{equation*}
g^{-1}=\left[\begin{array}{cccccc}d_{1}&	0&	0&	-b_{1}&	0&	-(d_{1}\mu'_{1}+b_{1}\lambda'_{1})a_3\\
		0&	d_{2}&	0&	0&	-b_{2}&	-(d_{2}\mu'_{2}+b_{2}\lambda'_{2})a_3\\
		-\lambda'_{1}&	-\lambda'_{2}&	1/a_3&	-\mu'_{1}&	-\mu'_{2}&	-\kappa\\
		-c_{1}&	0&	0&	a_{1}&	0&	(c_{1}\mu'_{1}+a_{1}\lambda'_{1})a_3\\
		0&	-c_{2}&	0&	0&	a_{2}&	(c_{2}\mu'_{2}+a_{2}\lambda'_{2})a_3\\
		0&	0&	0&	0&	0&	a_{3}
\end{array}\right].
\end{equation*}
We have
\begin{equation*}\begin{split}
 \Ad{g^{-1}}E_1&=(d_{1}^2-c_{1}^2)E_1+{\lambda'_{1}}^2E_{3}-d_{1}\lambda'_{1}E_{3,1}+2c_{1}d_{1}A_1-c_{1}\lambda'_{1}A_{3,1}-c_{1}^2W_{1},\\
 \Ad{g^{-1}}E_2&=(d_{2}^2-c_{2}^2)E_2+{\lambda'_{2}}^2E_{3}-d_{2}\lambda'_{2}E_{3,2}+2c_{2}d_{2}A_2-c_{2}\lambda'_{2}A_{3,2}-c_{2}^2W_2,\\
 \Ad{g^{-1}}A_{3,1}&=(2\mu'_{1}E_{3}+b_{1}E_{3,1}+a_{1}A_{3,1})/a_3,\\
 \Ad{g^{-1}}A_{3,2}&=(2\mu'_{2}E_{3}+b_{2}E_{3,2}+a_{2}A_{3,2})/a_3,\\
 \Ad{g^{-1}}W_{1}&=(-d_{1}^2+c_{1}^2-b_{1}^2+a_{1}^2)E_1+(-{\mu'_{1}}^2-{\lambda'_{1}}^2)E_{3}+(d_{1}\lambda'_{1}-b_{1}\mu'_{1})E_{3,1}\\\quad&+(-2c_{1}d_{1}-2a_{1}b_{1})A_1+(c_{1}\lambda'_{1}-a_{1}\mu'_{1})A_{3,1}+(c_{1}^2+a_{1}^2)W_1,\\
 \Ad{g^{-1}}W_2&=(-d_{2}^2+c_{2}^2-b_{2}^2+a_{2}^2)E_2+(-{\mu'_{2}}^2-{\lambda'_{2}}^2)E_{3}+(d_{2}\lambda'_{2}-b_{2}\mu'_{2})E_{3,2}\\\quad&+(-2c_{2}d_{2}-2a_{2}b_{2})A_{2}+(c_{2}\lambda'_{2}-a_{2}\mu'_{2})A_{3,2}+(c_{2}^2+a_{2}^2)W_{2}.
\end{split}\end{equation*}
Suppose that $\Ad{g^{-1}}$ induces the identity map of $\mathfrak{g}/\langle E_3\rangle$ onto itself.
Then we have
\begin{equation*}\begin{split}
d_1^2-c_1^2=1,\quad d_1\lambda'_1=c_1^2=0,\\
d_2^2-c_2^2=1,\quad d_2\lambda'_2=c_2^2=0,\\
b_1=0,\quad a_1/a_3=1,\\
b_2=0,\quad a_2/a_3=1,\\
c_1\lambda'_1-a_1\mu'_1=0,\quad c_1^2+a_1^2=1,\\
c_2\lambda'_2-a_2\mu'_2=0,\quad c_2^2+a_2^2=1.
\end{split}\end{equation*}
Hence it follows that $g^{-1}\in \exp\langle E_3\rangle\subset G$, which implies that $G/\exp\langle E_3\rangle $ has trivial center.

(b) By (a), we see that $G/\exp\langle E_3\rangle $ is linearlizable.
By Theorem \ref{Letarbitgr}, we conclude that the image of the subgroup $\exp\langle W_1,W_2\rangle$ of $G$ under the quotient map is a maximal compact subgroup of $G/\exp\langle E_3\rangle$.
Note that the image of a compact subgroup or a real split solvable Lie subgroup under a homomorphism of a Lie groups is also a compact subgroup or real split solvable Lie subgroup, respectively.

(c) The group $\exp \mathfrak{h}'_3\oplus\mathfrak{s}'_3\subset G$ is a topological product of $H_3(\mathbb{R})$ and $SL(2,\mathbb{R})$.
By (a), it follows that $\pi_1(\Int{\tilde{\mathfrak{g}}},e)=\mathbb{Z}^2/\mathbb{Z}=\mathbb{Z}$.
\end{proof}
In the case (vi), we have $\pi_1(\mathop{\mathrm{Int}}\tilde{\mathfrak{g}},e)=\mathbb{Z}$.
We may assume that $\mathfrak{k}\subset \langle E_2,E_3,E_{3,2},A_2,A_{3,2},W_1,\allowbreak
W_2\rangle/\langle E_2, E_3, E_{3,2},A_2, A_{3,2},W_2\rangle$, and
we then have $\dim \mathfrak{k}=0$.
Since $M$ is diffeomorphic to a coadjoint orbit, the group $\mathop{\mathrm{Int}}{\tilde{\mathfrak{g}}}$ acts transitively on $M$, and the isotropy subgroup $(\mathop{\mathrm{Int}}\tilde{\mathfrak{g}})_{m_0}$ at $m_0$ equals $\{e\}$.
Thus $\pi_1((\mathop{\mathrm{Int}}\tilde{\mathfrak{g}})_{m_0},e)=\{e\}$.
This contradicts that $M$ is simply connected.
Similarly, we have $\pi_1(\mathop{\mathrm{Int}}\tilde{\mathfrak{g}},e)=\mathbb{Z}^2$, $\pi_1((\Int{\tilde{\mathfrak{g}}})_{m_0},e)=\mathbb{Z}$ in the case (vii),
and we have $\pi_1(\mathop{\mathrm{Int}}\tilde{\mathfrak{g}},e)=\mathbb{Z}^2$, $\pi_1((\Int{\tilde{\mathfrak{g}}})_{m_0},e)=\mathbb{Z}$ in the case (viii).
These results contradict that $M$ is simply connected.
We obtain the following theorem.

\begin{theorem}\label{Everyirred2}
Every irreducible non-generic CS representation of $G$ is given by the composition of the external tensor product of a one-dimensional unitary representation of $\mathbb{R}_{>0}$ and a nontrivial irreducible highest weight representation of $SL(2,\mathbb{R})\times SL(2,\mathbb{R})$ with the map $G\simeq G'\rightarrow\mathbb{R}_{>0}\times SL(2,\mathbb{R})\times SL(2,\mathbb{R})$ given by
\begin{equation*}\begin{split}
G'\ni &\left[\begin{array}{cccccc}a_1&0&0&b_1&0&\mu_1'\\
0&a_2&0&0&b_2&\mu_2'\\
\lambda_1&\lambda_2&a_3&\mu_1&\mu_2&\kappa\\
c_1&0&0&d_1&0&-\lambda_1'\\
0&c_2&0&0&d_2&-\lambda_2'\\
0&0&0&0&0&a_3^{-1}
\end{array}\right]\\&\mapsto\left(a_3, \left[\begin{array}{cc}a_1&b_1\\c_1&d_1\end{array}\right],\left[\begin{array}{cc}a_2&b_2\\c_2&d_2\end{array}\right]\right)\in \mathbb{R}_{>0}\times SL(2,\mathbb{R})\times SL(2,\mathbb{R}).
\end{split}\end{equation*}
\end{theorem}
\section{Intertwining operators}\label{Intertwini}
For $n,n'\in\mathbb{Z}$, let $(\pi_{n,n'}, \mathcal{H}^{n,n'})$ be any irreducible unitary representation of $SL(2,\mathbb{R})\times SL(2,\mathbb{R})$ such that 
there exists $v\in(\mathcal{H}^{n,n'})^\infty\backslash\{0\}$ satisfying
\begin{equation*}
d\pi_{n,n'}(x,y)\,v=0\quad \mbox{for}\,(x,y) \in\mathbb{C}\left[\begin{array}{cc}-i&1\\1&i\end{array}\right]\times\mathbb{C}\left[\begin{array}{cc}-i&1\\1&i\end{array}\right]
\end{equation*}
and
\begin{equation*}
\pi_{n,n'}\left(\left[\begin{array}{cc}\cos \theta&-\sin\theta\\\sin\theta&\cos\theta\end{array}\right],\left[\begin{array}{cc}\cos \tau&-\sin\tau\\\sin\tau&\cos\tau\end{array}\right]\right)v=e^{i(n\theta+n'\tau)}v\quad(\theta,\tau\in\mathbb{R}).
\end{equation*}
Then the set of equivalence classes of irreducible highest weight representations of $SL(2,\mathbb{R})\times SL(2,\mathbb{R})$ is given by $\{[\pi_{n,n'}]; n,n'\in\mathbb{Z}\}$.
Let $\tilde{G}=\mathbb{R}_{>0}\times SL(2,\mathbb{R})\times SL(2,\mathbb{R})$.
For $\eta_3\in\mathbb{R}$, let $(\pi_{\eta_3,n,n'},\mathcal{H}^{n,n'})$ be the external tensor product of the one dimensional representation of $\mathbb{R}_{>0}$ given by $\mathbb{R}_{>0}\ni \gamma\mapsto \gamma^{2i\eta_3}\in\mathbb{C}^\times$ and the representation $\pi_{n,n'}$ of $SL(2,\mathbb{R})\times SL(2,\mathbb{R})$.
Composing with the map $G\rightarrow \tilde{G}$ given in Theorem \ref{Everyirred2}, we regard $\pi_{\eta_3,n,n'}$ as a representation of $G$.
Let $n,n'\in\mathbb{Z}_{\geq 0}$.
By \eqref{dpixiover}, we can take $\pi_{\xi(0,\eta_3,n,n')}|_{SL(2,\mathbb{R})\times SL(2,\mathbb{R})}$ to be the irreducible unitary representation $\pi_{n,n'}$, and hence $\pi_{\xi(0,\eta_3,n,n')}$ is unitarily equivalent with $\pi_{\eta_3,n,n'}$ as representations of $G$.
Therefore we get the following theorem.
\begin{theorem}\label{Thesetofun}
The set of unitary equivalence classes of irreducible non-generic CS representations of $G$ is given by \begin{equation*}
\{[\pi_{\eta_3,n,n'}];(\eta_3,n,n')\in\mathbb{R}\times\mathbb{Z}\times\mathbb{Z}\backslash\mathbb{R}\times\{0\}\times\{0\}\}.
\end{equation*}
For $(\eta_3,n,n')\in\mathbb{R}\times\mathbb{Z}_{\geq0}\times\mathbb{Z}_{\geq 0}$, we have $\pi_{\eta_3,n,n'}\simeq \pi_{\xi(0,\eta_3,n,n')}$ and $\pi_{-\eta_3,-n,-n'}\simeq \overline{\pi_{\xi(0,\eta_3,n,n')}}$.
\end{theorem}
We fix a triple $(\eta_3,n,n')$ with $\eta\in\mathbb{R}$ and $n,n'\in\mathbb{Z}_{\geq 0}$.
We shall give an explicit description of an intertwining operator between the unitary representations $\pi_{\eta_3,n,n'}$ and $ \pi_{\xi(0,\eta_3,n,n')}$ of $G$.
Using the realization of $G$ as a linear group in Section \ref{theholomor}, we shall define a holomorphic multiplier representation of $G$.
Let $m:G\times \mathcal{D}_5\rightarrow\mathbb{C}^\times$ be the holomorphic multiplier given by 
\begin{equation*}\begin{split}
m(g,z)=(c_1z^1+d_1)^n(c_2z^2+d_2)^{n'}{a_3}^{2i\eta_3}\quad
(g\in G,z\in\mathcal{D}_5),
\end{split}\end{equation*}
and let $\tau_m$ be the holomorphic multiplier representation given by
\begin{equation*}\begin{split}
\tau_m(g)f(z)=m(g^{-1},z)^{-1}f(g^{-1}z)\quad
(g\in G,z\in \mathcal{D}_5,f\in\mathcal{O}(\mathcal{D}_5)).
\end{split}\end{equation*}
Then $\pi_{\xi(0,\eta_3,n,n')}$ can be considered as a unitarization of $\tau_m$.

Next we see a natural holomorphic multiplier representation of $G$ in which $\pi_{\eta_3,n,n'}$ is realized.
Let $\mathcal{D}_1$ be the unit disc in $
\mathbb{C}$, and
let $\tilde{m}:\tilde{G}\times \mathcal{D}_1\times\mathcal{D}_1\rightarrow\mathbb{C}^\times$ be the holomorphic multiplier given by 
\begin{equation*}\begin{split}
\tilde{m}((\gamma, g_1,g_2),(w^1,w^2))=(c_1w^1+d_1)^n(c_2w^2+d_2)^{n'}\gamma^{2i\eta_3}\\
((\gamma, g_1,g_2)\in \tilde{G},(w^1,w^2)\in\mathcal{D}_1\times\mathcal{D}_1),
\end{split}\end{equation*}
where $g_i=\left[\begin{array}{cc}a_i&b_i\\c_i&d_i\end{array}\right]\in SL(2,\mathbb{R})$ for $i=1,2$.
We denote by $\mathcal{D}_1\ni w^i\mapsto g_iw^i\in\mathcal{D}_1$ the action of $SL(2,\mathbb{R})$ by linear fractional transformations for $i=1,2$.
Then we can define the following holomorphic multiplier representation $\tau_{\tilde{m}}$ of $\tilde{G}$ on the space $\mathcal{O}(\mathcal{D}_1\times\mathcal{D}_1)$ of holomorphic functions on $\mathcal{D}_1\times\mathcal{D}_1$:
\begin{equation*}\begin{split}
\tau_{\tilde{m}}(g)f(w^1,w^2)=m(g^{-1},(w^1,w^2))^{-1}f(g_1^{-1}w^1,g_2^{-1}w^2)\\
(g=(\gamma, g_1,g_2)\in \tilde{G},(w^1,w^2)\in\mathcal{D}_1\times\mathcal{D}_1,f\in\mathcal{O}(\mathcal{D}_1\times\mathcal{D}_1)).
\end{split}\end{equation*}
We regard $\tau_{\tilde{m}}$ as a representation of $G$ which $\exp\mathfrak{h}_5\subset G$ acts by the trivial representation.
Then we have the following theorem.
\begin{theorem}\label{Themapmath}
The map $F:\mathcal{O}(\mathcal{D}_1\times\mathcal{D}_1)\ni f\mapsto F_f\in\mathcal{O}(\mathcal{D}_5)$ defined by $F_f(z)=f(z^1,z^2)\,(z\in\mathcal{D}_5)$ intertwines $\tau_{\tilde{m}}$ with $\tau_m$, and hence $F$ gives rise to an intertwining operator between the unitarizations.
\end{theorem}

\end{document}